\documentclass[12pt,bezier]{article}
\usepackage{times}
\usepackage{booktabs}
\usepackage{pifont}
\usepackage{floatrow}
\usepackage{caption}
\usepackage{mathrsfs}
\usepackage[fleqn]{amsmath}
\usepackage{amsfonts,amsthm,amssymb,mathrsfs,bbding}
\usepackage{txfonts}
\usepackage{graphics,multicol}
\usepackage{graphicx}
\usepackage{color}
\usepackage{amssymb}
\usepackage{caption}
\usepackage{cite}
\usepackage{latexsym,bm}
\usepackage{indentfirst}
\usepackage{color}
\usepackage[colorlinks=true,anchorcolor=blue,filecolor=blue,linkcolor=blue,urlcolor=blue,citecolor=blue]{hyperref}
\usepackage{extarrows}
\usepackage{cite}
\usepackage{latexsym,bm}
\usepackage{mathtools}
\evensidemargin=\oddsidemargin\topmargin=-1.5cm

\newtheorem{thm}{Theorem}[section]

\newtheorem{prob}{Problem}[section]
\newtheorem{claim}{Claim}

\newtheorem{lem}{Lemma}[section]

\theoremstyle{definition}

\addtocounter{section}{0}

\begin{document}
\title{Toughness in regular graphs from eigenvalues\footnote{Supported by the National Natural Science Foundation of China
{(Nos. 12371361 and 12471331)}, Distinguished Youth Foundation of Henan Province {(No. 242300421045)} and Zhengzhou University Young Student Basic Research Projects (No. ZDBJ202521).}}
\author{\bf Ruifang Liu$^{a}$, {\bf Ao Fan$^{a}$}\thanks{Corresponding author.
E-mail addresses: rfliu@zzu.edu.cn (R. Liu), fanaozzu@163.com (A. Fan), jlshu@shnu.edu.cn (J. Shu).}, {\bf Jinlong Shu$^{b}$}\\
{\footnotesize $^a$ School of Mathematics and Statistics, Zhengzhou University, Zhengzhou, Henan 450001, China} \\
{\footnotesize $^b$ School of Finance and Business, Shanghai Normal University, Shanghai 200233, China}}
\date{}

\maketitle
{\flushleft\large\bf Abstract}
The {\it toughness} $\tau(G)=\mathrm{min}\{\frac{|S|}{c(G-S)}: S~\mbox{is a vertex cut in}~G\}$ for $G\ncong K_n,$ which was initially proposed by Chv\'{a}tal in 1973. A graph $G$ is called {\it $t$-tough} if $\tau(G)\geq t.$
Let $\lambda_i(G)$ be the $i$-th largest eigenvalue of the adjacency matrix of a graph $G$.
In 1996, Brouwer conjectured that $\tau(G)\geq\frac{d}{\lambda}-1$ for a connected $d$-regular graph $G,$ where $\lambda=\mathrm{max}\{|\lambda_2|, |\lambda_n|\}.$ Gu [SIAM J. Discrete Math. 35 (2021) 948-952] completely confirmed this conjecture. From Brouwer and Gu's result $\tau(G)\geq\frac{d}{\lambda}-1,$ we know that if $G$ is a connected $d$-regular graph and $\lambda\leq\frac{bd}{b+1}$, then $\tau(G)\geq\frac{1}{b}$ for an integer $b\geq1.$
Inspired by the above result and utilizing typical spectral techniques and graph construction methods from Cioab\u{a} et al. [J. Combin. Theory Ser. B 99 (2009) 287-297],
we prove that if $G$ is a connected $d$-regular graph and $\lambda_2(G)<\phi(d,b)$, then $\tau(G)\geq\frac{1}{b}$. Meanwhile, we construct graphs implying that the upper bound on $\lambda_2(G)$ is best possible. Our theorem strengthens the result of Chen et al. [Discrete Math. 348 (2025) 114404].
Finally, we also prove an upper bound of $\lambda_{b+1}(G)$ to guarantee a connected $d$-regular graph to be $\frac{1}{b}$-tough.

\begin{flushleft}
\end{flushleft}

\section{Introduction}

Let $G$ be an undirected, simple and connected graph with vertex set $V(G)$ and edge set $E(G)$. The order and size of $G$ are denoted by $|V(G)|=n$ and $|E(G)|=e(G)$.
We denote by $c(G)$ the number of components of $G$. For a vertex subset $S\subseteq V(G)$, let $G[S]$ and $|S|$ be the subgraph of $G$ induced by $S$ and the size of $S$, respectively.
Let $G_1$ and $G_2$ be two vertex-disjoint graphs. We denote by $G_{1}\cup G_{2}$ the disjoint union of $G_1$ and $G_2$. The join $G_1\vee G_2$ is the graph obtained from $G_1\cup G_2$ by adding all possible edges between $V(G_1)$ and $V(G_2)$. Let $\overline{G}$ be the complement of $G$. The {\it toughness} of a graph $G$ $$\tau(G)=\mathrm{min}\{\frac{|S|}{c(G-S)}: S~\mbox{is a vertex cut in}~G\}$$ for $G\ncong K_n.$
For undefined notions and symbols, one can refer to \cite{Bondy2008}.

Given a graph $G$ of order $n$, the adjacency matrix of $G$ is the 0-1 matrix $A(G)=(a_{ij})_{n\times n}$, where $a_{ij}=1$ if $v_{i}\sim v_{j}$ and $a_{ij}=0$ otherwise. Note that $A(G)$ is a real nonnegative symmetric matrix. Hence its eigenvalues are real, which can be arranged in non-increasing order as $\lambda_{1}(G)\geq \lambda_{2}(G) \geq \cdots \geq \lambda_{n}(G).$ The largest eigenvalue of $A(G)$, denoted by $\rho(G)$, is called the {\it spectral radius} of $G$. If $G$ is $d$-regular, then it is easy to see that $\rho(G)=d$ and that $\lambda_{2}(G)<d$ if and only if $G$ is connected.

There have been numerous significant results on studying the combinatorial properties of regular graphs using eigenvalues. Brouwer and Haemers\cite{Brouwer2005} proved that if $G$ is a connected $d$-regular graph and
\begin{eqnarray*}
\begin{split}
\lambda_3(G)\leq \left \{
\begin{array}{ll}
d-1+\frac{3}{d+1}  & \mbox{if $d$ is even},\\
d-1+\frac{3}{d+2}  & \mbox{if $d$ is odd},
\end{array}
\right.
\end{split}
\end{eqnarray*}
then $G$ contains a perfect matching. Cioab\u{a} and Gregory\cite{Cio2007} generalized the above result to the matching number in a connected $d$-regular graph $G$. Let $\theta$ be the largest root of $x^3-x^2-6x+2=0.$
Cioab\u{a} et al.\cite{Cio2009} further proved that if $G$ is a connected $d$-regular graph and
\begin{eqnarray*}
\begin{split}
\lambda_3(G)< \left \{
\begin{array}{ll}
\theta=2.85577\ldots  & \mbox{if $d=3$,}\\
\frac{d-2+\sqrt{d^2+12}}{2}  & \mbox{if $d\geq4$ is even},\\
\frac{d-3+\sqrt{d^2+2d+17}}{2}  & \mbox{if $d\geq5$ is odd,}
\end{array}
\right.
\end{split}
\end{eqnarray*}
then $G$ contains a perfect matching.
Later, Lu et al.\cite{Lu2010} gave a sufficient condition in terms of $\lambda_{3}(G)$ for the existence of an odd $[1,b]$-factor in a connected regular graph. Kim et al. \cite{Kim2020} improved the above result. Using the Tutte's $k$-Factor Theorem\cite{Tutte1947}, Lu\cite{LuH2010,Lu2012} presented sufficient conditions in terms of $\lambda_{3}(G)$ for the existence of a $k$-factor in a connected regular graph. In 2022, O\cite{O2022} proved upper bounds (in terms of $a$, $b$ and $d$) on certain eigenvalues (in terms of $a$, $b$, $d$ and $h$) in an $h$-edge-connected $d$-regular graph $G$ to guarantee the existence of an even (or odd) $[a, b]$-factor.

The investigation into the interplay between toughness and graph eigenvalues was pioneered by Alon\cite{Alon1995} who established that for any connected $d$-regular graph $G$, the toughness satisfies $\tau(G)>\frac{1}{3}\left(\frac{d^2}{d\lambda+\lambda^2}\right)$, where $\lambda=\mathrm{max}\{|\lambda_2(G)|, |\lambda_n(G)|\}.$
Concurrently, Brouwer\cite{Brouwer1995} independently proved that $\tau(G)>\frac{d}{\lambda}-2,$ and further conjectured that $\tau(G)>\frac{d}{\lambda}-1$ in\cite{Brouwer1996}. Subsequently, Gu\cite{Gu2021} improved Brouwer's result to $\tau(G)\geq\frac{d}{\lambda}-\sqrt{2},$ and further completely confirmed the conjecture\cite{GuDM}. Very recently, Fan et al.\cite{FanLin2023} provided spectral radius conditions for a graph to be $t$-tough, where $t$ is a positive integer. The interplay between Laplacian eigenvalues and graph toughness has been extensively investigated in the literature, with key references including \cite{Gu20222,Liu2010}.

Cioab\u{a} and Wong\cite{Cio2014} proved a best possible upper bound of $\lambda_2(G)$ for a connected $d$-regular graph to be $1$-tough.
\begin{thm}[Cioab\u{a} and Wong\cite{Cio2014}]\label{thm1}
Let $G$ be a connected $d$-regular graph.
If
\begin{eqnarray*}
\begin{split}
\lambda_2(G)< \left \{
\begin{array}{ll}
\frac{d-2+\sqrt{d^2+8}}{2}  & \mbox{if $d$ is odd,}\\
\frac{d-2+\sqrt{d^2+12}}{2}& \mbox{if $d$ is even},
\end{array}
\right.
\end{split}
\end{eqnarray*}
then $\tau(G)\geq1$.
\end{thm}

Incorporating the toughness and eigenvalues of a graph, Chen et al.\cite{Chen2025} provided a sufficient condition based on $\lambda_2(G)$ for a connected $d$-regular graph to be $\frac{1}{b}$-tough, where $d\leq b^2+b$ and $b\geq2$.

\begin{thm}[Chen et al.\cite{Chen2025}]\label{thm2}
Let $G$ be a connected $d$-regular graph with integers $d\leq b^2+b$ and $b\geq2$.
If
\begin{eqnarray*}
\begin{split}
\lambda_2(G)< \left \{
\begin{array}{ll}
\frac{d-2+\sqrt{d^2+4(d-b)+8}}{2}  & \mbox{if $b$ is odd,}\\
\frac{d-2+\sqrt{d^2+4(d-b)+4}}{2}& \mbox{if $b$ is even},
\end{array}
\right.
\end{split}
\end{eqnarray*}
then $\tau(G)\geq\frac{1}{b}$.
\end{thm}

In this paper, we prove a best possible upper bound of $\lambda_2(G)$ for a connected $d$-regular graph to be $\frac{1}{b}$-tough, which strengthes the result of Theorem \ref{thm2}.
Let $d$ and $b$ be two positive integers. Define
\begin{eqnarray*}
\begin{split}
\phi(d,b)= \left \{
\begin{array}{ll}
\alpha_d & \mbox{if $\lceil\frac{d}{b}\rceil\leq2$,}\\
\frac{d-2+\sqrt{d^2+4d+8-4\lceil\frac{d}{b}\rceil}}{2}  & \mbox{if $\lceil\frac{d}{b}\rceil\geq3$ is odd,}\\
\frac{d-3+\sqrt{d^2+6d+13-4\lceil\frac{d}{b}\rceil}}{2}& \mbox{if $\lceil\frac{d}{b}\rceil\geq3$ is even and $d$ is odd},\\
\frac{d-2+\sqrt{d^2+4d+12-4\lceil\frac{d}{b}\rceil}}{2}& \mbox{if $\lceil\frac{d}{b}\rceil\geq3$ and $d$ are even,}\\
\end{array}
\right.
\end{split}
\end{eqnarray*}
where $\alpha_d$ is the largest root of the equation $x^3-(d-2)x^2-2dx+d-1=0$.

\begin{thm}\label{thm3}
Let $G$ be a connected $d$-regular graph.
If $$\lambda_2(G)<\phi(d,b),$$
then $\tau(G)\geq\frac{1}{b}$, where $b\geq1$ is an integer.
\end{thm}

Our result generalizes Theorem \ref{thm1} from $b=1$ to general $b.$ In Section $4$, we construct graphs to show that the upper bound on $\lambda_2(G)$ in Theorem \ref{thm3} is best possible. Furthermore, we in Section $5$, provide a sufficient condition in terms of $\lambda_{b+1}(G)$ to ensure that a connected $d$-regular graph is $\frac{1}{b}$-tough for $b\geq 1$. Let
\begin{eqnarray*}
\begin{split}
\psi(d,b)= \left \{
\begin{array}{ll}
\alpha_d & \mbox{if $d\leq b+1$,}\\
\frac{d-2+\sqrt{d^2+4b+4}}{2}& \mbox{if $d\geq b+2$ and $b$ have the same parity,}\\
\frac{d-3+\sqrt{d^2+4b+2d+9}}{2}  & \mbox{if $d\geq b+2$ is odd and $b$ is even,}\\
\frac{d-2+\sqrt{d^2+4b+8}}{2}& \mbox{if $d\geq b+2$ is even and $b$ is odd.}\\
\end{array}
\right.
\end{split}
\end{eqnarray*}

\begin{thm}\label{thm4}
Let $G$ be a connected $d$-regular graph.
If $$\lambda_{b+1}(G)<\psi(d,b),$$
then $\tau(G)\geq\frac{1}{b}$, where $b\geq1$ is an integer.
\end{thm}

In fact, our result also generalizes Theorem \ref{thm1} for $b=1.$

\section{Tools}
Consider an $n\times n$ real symmetric matrix
\[
M=\left(\begin{array}{ccccccc}
M_{1,1}&M_{1,2}&\cdots &M_{1,m}\\
M_{2,1}&M_{2,2}&\cdots &M_{2,m}\\
\vdots& \vdots& \ddots& \vdots\\
M_{m,1}&M_{m,2}&\cdots &M_{m,m}\\
\end{array}\right),
\]
whose rows and columns are partitioned according to a partitioning
$X_{1}, X_{2},\ldots ,X_{m}$ of $\{1,2,\ldots, n\}$. The \emph{quotient matrix}
$R$ of the matrix $M$ is the $m\times m$ matrix whose entries are the
average row sums of the blocks $M_{i,j}$ of $M$. The partition is \emph{equitable}
if each block $M_{i,j}$ of $M$ has constant row (and column) sum.

\begin{lem}[Brouwer and Haemers \cite{Brouwer2011}, Godsil and Royle \cite{Godsil2001}, Haemers \cite{Haemers1995}]\label{le1}
Let $M$ be a real symmetric matrix and $R(M)$ be its quotient matrix. Then the eigenvalues of every quotient matrix of $M$ interlace the ones of $M$. If $R(M)$ is equitable, then the eigenvalues of $R(M)$ are eigenvalues of $M$. Furthermore, if $M$ is nonnegative and irreducible, then the spectral radius of $R(M)$ equals to the spectral radius of $M$.
\end{lem}

\begin{lem}[Brouwer and Haemers \cite{Brouwer2011}, Godsil and Royle \cite{Godsil2001}]\label{le2}
If $H$ is an induced subgraph of a graph $G$, then $\lambda_i(G)\geq\lambda_i(H)$ for all $i\in\{1, \ldots, n\}.$
\end{lem}

\begin{lem}[Cvetkovi\'{c} et al.\cite{Doob1980}]\label{le3}
Let $G$ be a connected graph with $n$ vertices and $m$ edges. Then
$$\rho(G)\geq\frac{2m}{n}$$
with equality if and only if $G$ is a regular graph.
\end{lem}

\begin{lem}\label{le4}
$\rho((K_1\cup K_2)\vee\overline{\frac{d-1}{2}K_2})=\alpha_d,$ where $d\geq 1$ is an odd integer.
\end{lem}
\begin{proof}
Let $G=(K_1\cup K_2)\vee\overline{\frac{d-1}{2}K_2}$ . Consider the vertex partition $\{V(K_1), V(K_2), V(\overline{\frac{d-1}{2}K_2})\}$ of $G$.
The quotient matrix $R(A(G))$ on the vertex partition equals
$$R(A(G))=\left[
\begin{array}{ccc}
0&0&d-1\\
0&1&d-1\\
1&2&d-3
\end{array}
\right].
$$
The characteristic polynomial of $R(A(G))$ is $x^3-(d-2)x^2-2dx+d-1.$
By Lemma \ref{le1}, $\rho(R(A(G)))=\rho((K_1\cup K_2)\vee\overline{\frac{d-1}{2}K_2})=\alpha_d.$
\end{proof}

In order to prove Lemmas \ref{le6} and \ref{le7}, we first give a crucial result.

\begin{lem}\label{le5}
$d-\frac{1}{d+4}>\alpha_d$ for $d\geq 3.$
\end{lem}
\begin{proof}
Let $f(x)=x^3-(d-2)x^2-2dx+d-1.$ Recall that $\alpha_d$ is the largest root of the equation $f(x)=0.$ Then $f'(x)=3x^2-(2d-4)x-2d$ and the symmetry axis of $f'(x)$ is
$x=\frac{d-2}{3}<d-\frac{1}{d+4}.$ For $x \geq d-\frac{1}{d+4}$, we have $$f'(x)\geq f'(d-\frac{1}{d+4})=\frac{d^4+10d^3+28d^2+12d-13}{(d+4)^2}>0.$$
This implies that $f(x)$ is increasing with respect to $x\geq d-\frac{1}{d+4}.$ Since $d\geq 3,$
$$f(d-\frac{1}{d+4})=\frac{d^3+6d^2-6d-57}{(d+4)^3}>0,$$
which implies that $d-\frac{1}{d+4}>\alpha_d.$
\end{proof}

Let $e(S,H)$ be the number of edges in $G$ between $S$ and $H$. For simplicity, we define $n_H=|V(H)|$ and $m_H=|E(H)|$.
\begin{lem}\label{le6}
Let $G$ be a connected $d$-regular graph and $b\geq 1$ be an integer with $\lceil\frac{d}{b}\rceil \geq 3.$ Let $H$ be a component of $G-S$ such that $e(S,H)<\lceil\frac{d}{b}\rceil$, where $S\subseteq V(G)$ is not an empty set. If $\rho(H)\leq\rho(H')$ for every component $H'$ of $G-S$ with $e(S,H')<\lceil\frac{d}{b}\rceil$ and $\rho(H)\leq \phi(d,b)$, then we have
\begin{eqnarray*}
\begin{split}
n_H= \left \{
\begin{array}{ll}
d+2  & \mbox{if $d$ and $n_H$ have the same parity,}\\
d+1  & \mbox{otherwise,}
\end{array}
\right.
\end{split}
\end{eqnarray*}
and
\begin{eqnarray*}
\begin{split}
2m_H= \left \{
\begin{array}{ll}
d(d+2)-\lceil\frac{d}{b}\rceil+2 & \mbox{if $d$, $n_H$ and $\lceil\frac{d}{b}\rceil$ have the same parity,}\\
d(d+2)-\lceil\frac{d}{b}\rceil+1 & \mbox{if $d$, $n_H$ are even and $\lceil\frac{d}{b}\rceil$ is odd or $d$, $n_H$} \\& \mbox{are odd and $\lceil\frac{d}{b}\rceil$ is even,}\\
d(d+1)-\lceil\frac{d}{b}\rceil+2 & \mbox{if $d$, $n_H$ have different parity and $\lceil\frac{d}{b}\rceil$ is even},\\
d(d+1)-\lceil\frac{d}{b}\rceil+1  & \mbox{if $d$, $n_H$ have different parity and $\lceil\frac{d}{b}\rceil$ is odd.}
\end{array}
\right.
\end{split}
\end{eqnarray*}
\end{lem}

\begin{proof}
Note that $e(S,H)<\lceil\frac{d}{b}\rceil$ and $G$ is a connected $d$-regular graph. Then we have
\begin{eqnarray}\label{eq00}
n_H(n_H-1)\geq2m_H=dn_H-e(S,H)\geq dn_H-\Big\lceil\frac{d}{b}\Big\rceil+1.
\end{eqnarray}
We claim that $n_H\geq 2.$ In fact, if $n_H=1,$ then $e(S,H)=d\geq \lceil\frac{d}{b}\rceil,$ a contradiction. Hence $n_H\geq d+\frac{d-\lceil\frac{d}{b}\rceil+1}{n_H-1}>d,$ that is, $n_H\geq d+1.$ This implies that
\begin{eqnarray*}
\begin{split}
n_H\geq \left \{
\begin{array}{ll}
d+2  & \mbox{if $d$ and $n_H$ have the same parity,}\\
d+1  & \mbox{otherwise.}
\end{array}
\right.
\end{split}
\end{eqnarray*}
Suppose that
\begin{eqnarray*}
\begin{split}
n_H> \left \{
\begin{array}{ll}
d+2  & \mbox{if $d$ and $n_H$ have the same parity,}\\
d+1  & \mbox{otherwise.}
\end{array}
\right.
\end{split}
\end{eqnarray*}

Assume that $\lceil\frac{d}{b}\rceil$ is odd. Recall that $\phi(d,b)=\frac{d-2+\sqrt{d^2+4d+8-4\lceil\frac{d}{b}\rceil}}{2}.$
If $d$ and $n_H$ are odd, then $n_H\geq d+4.$ By (\ref{eq00}), we have $2m_H\geq dn_H-\lceil\frac{d}{b}\rceil+2$. According to Lemma \ref{le3}, we have
$\rho(H)\geq\frac{2m_H}{n_H}\geq d-\frac{\lceil\frac{d}{b}\rceil-2}{n_H}\geq d-\frac{\lceil\frac{d}{b}\rceil-2}{d+4}
>\frac{d-2+\sqrt{d^2+4d+8-4\lceil\frac{d}{b}\rceil}}{2}.$
If $d$ and $n_H$ are even, then $n_H\geq d+4$ and $2m_H\geq dn_H-\lceil\frac{d}{b}\rceil+1$. So we have
$\rho(H)\geq\frac{2m_H}{n_H}
\geq d-\frac{\lceil\frac{d}{b}\rceil-1}{d+4}
>\frac{d-2+\sqrt{d^2+4d+8-4\lceil\frac{d}{b}\rceil}}{2}.$
If $d$ and $n_H$ have different parity, then $n_H\geq d+3$ and $2m_H\geq dn_H-\lceil\frac{d}{b}\rceil+1$. Hence we obtain that
$\rho(H)\geq\frac{2m_H}{n_H}\geq d-\frac{\lceil\frac{d}{b}\rceil-1}{d+3}
>\frac{d-2+\sqrt{d^2+4d+8-4\lceil\frac{d}{b}\rceil}}{2}.$

Suppose that $\lceil\frac{d}{b}\rceil$ is even and $d$ is odd. Then $\phi(d,b)=\frac{d-3+\sqrt{d^2+6d+13-4\lceil\frac{d}{b}\rceil}}{2}.$
If $n_H$ is odd, then $n_H\geq d+4$ and $2m_H\geq dn_H-\lceil\frac{d}{b}\rceil+1$. So we have
$\rho(H)\geq\frac{2m_H}{n_H}
\geq d-\frac{\lceil\frac{d}{b}\rceil-1}{d+4}
>\frac{d-3+\sqrt{d^2+6d+13-4\lceil\frac{d}{b}\rceil}}{2}.$
If $n_H$ is even, then $n_H\geq d+3$
and $2m_H\geq dn_H-\lceil\frac{d}{b}\rceil+2$. Hence we have $\rho(H)\geq\frac{2m_H}{n_H}\geq d-\frac{\lceil\frac{d}{b}\rceil-2}{d+3}
>\frac{d-3+\sqrt{d^2+6d+13-4\lceil\frac{d}{b}\rceil}}{2}.$

Assume that $\lceil\frac{d}{b}\rceil$ and $d$ are even. Recall that  $\phi(d,b)=\frac{d-2+\sqrt{d^2+4d+12-4\lceil\frac{d}{b}\rceil}}{2}.$
If $n_H$ is odd, then $n_H\geq d+3$
and $2m_H\geq dn_H-\lceil\frac{d}{b}\rceil+2$. So we have
$\rho(H)\geq\frac{2m_H}{n_H}\geq d-\frac{\lceil\frac{d}{b}\rceil-2}{d+3}
>\frac{d-2+\sqrt{d^2+4d+12-4\lceil\frac{d}{b}\rceil}}{2}.$
If $n_H$ is even, then $n_H\geq d+4$ and $2m_H\geq dn_H-\lceil\frac{d}{b}\rceil+2$. Therefore, we have $\rho(H)\geq\frac{2m_H}{n_H}\geq d-\frac{\lceil\frac{d}{b}\rceil-2}{d+4}
>\frac{d-2+\sqrt{d^2+4d+12-4\lceil\frac{d}{b}\rceil}}{2}.$

In the above cases, we always have $\rho(H)>\phi(d,b),$ a contradiction.
Hence we have
\begin{eqnarray*}
\begin{split}
n_H= \left \{
\begin{array}{ll}
d+2  & \mbox{if $d$ and $n_H$ have the same parity,}\\
d+1  & \mbox{otherwise.}
\end{array}
\right.
\end{split}
\end{eqnarray*}

By (\ref{eq00}), we know that $2m_H\geq dn_H-\lceil\frac{d}{b}\rceil+1.$
Suppose that $2m_H>dn_H-\lceil\frac{d}{b}\rceil+1,$ that is,
\begin{eqnarray*}
\begin{split}
2m_H> \left \{
\begin{array}{ll}
d(d+2)-\lceil\frac{d}{b}\rceil+2 & \mbox{if $d$, $n_H$ and $\lceil\frac{d}{b}\rceil$ have the same parity,}\\
d(d+2)-\lceil\frac{d}{b}\rceil+1 & \mbox{if $d$, $n_H$ are even and $\lceil\frac{d}{b}\rceil$ is odd or $d$, $n_H$} \\& \mbox{are odd and $\lceil\frac{d}{b}\rceil$ is even,}\\
d(d+1)-\lceil\frac{d}{b}\rceil+2 & \mbox{if $d$, $n_H$ have different parity and $\lceil\frac{d}{b}\rceil$ is even},\\
d(d+1)-\lceil\frac{d}{b}\rceil+1  & \mbox{if $d$, $n_H$ have different parity and $\lceil\frac{d}{b}\rceil$ is odd.}
\end{array}
\right.
\end{split}
\end{eqnarray*}

Assume that $\lceil\frac{d}{b}\rceil$ is odd. Recall that $\phi(d,b)=\frac{d-2+\sqrt{d^2+4d+8-4\lceil\frac{d}{b}\rceil}}{2}.$
If $d$ and $n_H$ are odd, then $n_H=d+2$ and $2m_H\geq d(d+2)-\lceil\frac{d}{b}\rceil+4$. So we have $\rho(H)\geq\frac{2m_H}{n_H}\geq d-\frac{\lceil\frac{d}{b}\rceil-4}{d+2}
>\frac{d-2+\sqrt{d^2+4d+8-4\lceil\frac{d}{b}\rceil}}{2}.$
If $d$ and $n_H$ are even, then $n_H=d+2$ and $2m_H\geq d(d+2)-\lceil\frac{d}{b}\rceil+3$. Hence we obtain that $\rho(H)\geq\frac{2m_H}{n_H} \geq d-\frac{\lceil\frac{d}{b}\rceil-3}{d+2}
>\frac{d-2+\sqrt{d^2+4d+8-4\lceil\frac{d}{b}\rceil}}{2}.$
If $d$ and $n_H$ have different parity, then $n_H=d+1$ and $2m_H\geq d(d+1)-\lceil\frac{d}{b}\rceil+3$. Hence we have
$\rho(H)\geq\frac{2m_H}{n_H}\geq d-\frac{\lceil\frac{d}{b}\rceil-3}{d+1}
>\frac{d-2+\sqrt{d^2+4d+8-4\lceil\frac{d}{b}\rceil}}{2}.$

Suppose that $\lceil\frac{d}{b}\rceil$ is even and $d$ is odd. Then $\phi(d,b)=\frac{d-3+\sqrt{d^2+6d+13-4\lceil\frac{d}{b}\rceil}}{2}.$
If $n_H$ is odd, then $n_H=d+2$ and $2m_H\geq d(d+2)-\lceil\frac{d}{b}\rceil+3$. Hence we obtain that $\rho(H)\geq\frac{2m_H}{n_H}\geq d-\frac{\lceil\frac{d}{b}\rceil-3}{d+2}
>\frac{d-3+\sqrt{d^2+6d+13-4\lceil\frac{d}{b}\rceil}}{2}.$
If $n_H$ is even, then $n_H=d+1$
and $2m_H\geq d(d+1)-\lceil\frac{d}{b}\rceil+4$. So we have
$\rho(H)\geq\frac{2m_H}{n_H}\geq d-\frac{\lceil\frac{d}{b}\rceil-4}{d+1}
>\frac{d-3+\sqrt{d^2+6d+13-4\lceil\frac{d}{b}\rceil}}{2}.$

Assume that $\lceil\frac{d}{b}\rceil$ and $d$ are even. Recall that  $\phi(d,b)=\frac{d-2+\sqrt{d^2+4d+12-4\lceil\frac{d}{b}\rceil}}{2}.$
If $n_H$ is odd, then $n_H=d+1$
and $2m_H\geq d(d+1)-\lceil\frac{d}{b}\rceil+4$. Hence we have
$\rho(H)\geq\frac{2m_H}{n_H}\geq d-\frac{\lceil\frac{d}{b}\rceil-4}{d+1}
>\frac{d-2+\sqrt{d^2+4d+12-4\lceil\frac{d}{b}\rceil}}{2}.$
If $n_H$ is even, then $n_H=d+2$ and $2m_H\geq d(d+2)-\lceil\frac{d}{b}\rceil+4$. So we have
$\rho(H)\geq\frac{2m_H}{n_H}\geq d-\frac{\lceil\frac{d}{b}\rceil-4}{d+2}
>\frac{d-2+\sqrt{d^2+4d+12-4\lceil\frac{d}{b}\rceil}}{2}.$

These above cases always contradict $\rho(H)\leq\phi(d,b)$. So we complete the proof.
\end{proof}

\section{Proof of Theorem \ref{thm3}}

\medskip
\noindent  \textbf{Proof of Theorem \ref{thm3}.}
Suppose that a connected $d$-regular graph $G$ not a $\frac{1}{b}$-tough graph. There exists some subset $S\subseteq V(G)$ such that $c(G-S)\geq b|S|+1$. We choose $|S|$ to be as small as possible. According to the definition of toughness, we know that $S$ is not an empty set. Hence $|S|\geq1$. Let $|S|=s$ and $c(G-S)=q$. Then $q\geq bs+1$. Let $H_1, H_2, \ldots, H_q$ be the components of $G-S,$ and let $e(S,H_i)$ be the number of edges in $G$ between $S$ and $H_i$. Note that $G$ is connected. It is obvious that $e(S,H_i)\geq1$ and $q\leq\sum_{i=1}^{q}e(S,H_i)\leq sd.$

\begin{claim}\label{claim0}
$\lceil\frac{d}{b}\rceil\geq2$.
\end{claim}
\begin{proof}
Every vertex in $S$ must be adjacent to at least one connected component in $\{ H_1, H_2, \ldots, H_q \}$. Otherwise, there exists some vertex $ u\in S$ such that all neighbors of $u$ are contained in $S$. Consider the subset $S'= S\setminus \{u\}$, then $c(G-S')=q+1$, which contradicts the minimality of $S$. Recall that $q\geq bs+1$ and $s\geq1$.
Consequently, there exists a vertex $v \in S$ satisfying
$$d\geq d_{G-S}(v)\geq\frac{q}{s}\geq\frac{bs+1}{s}>b,$$
and hence $d\geq b+1.$ Then $\lceil\frac{d}{b}\rceil\geq2.$ \end{proof}

\begin{claim}\label{claim0.1}
If $\lceil\frac{d}{b}\rceil=2$, then $d$ is odd.
\end{claim}
\begin{proof}
Suppose that $d$ is even. Recall that there exists some subset $S\subseteq V(G)$ such that $q\geq bs+1$. Clearly, $2|E(H_i)|=d|V(H_i)|-e(S,H_i)$ for $i\in[1,q].$ Since $d$ is even, $e(S,H_i)\geq1$ is even. Then we have $$sd\geq\sum_{i=1}^{q}e(S,H_i)\geq2q>\Big\lceil\frac{d}{b}\Big\rceil bs \geq sd,$$ a contradiction.
\end{proof}
\begin{claim}\label{claim1}
There are at least two components, says $H_1$, $H_2$, such that $e(S,H_i)<\lceil\frac{d}{b}\rceil$ for $i\in\{1,2\}$.
\end{claim}
\begin{proof}
Assume to the contrary that there are at most one such component in $G-S.$ Since $G$ is $d$-regular, we have
$$sd\geq \sum_{i=1}^{q}e(S,H_i)\geq(q-1)\Big\lceil\frac{d}{b}\Big\rceil+1>bs\Big\lceil\frac{d}{b}\Big\rceil\geq sd,$$
which is a contradiction.
\end{proof}

Assume that $H$ is a component of $G-S$ such that $e(S,H)<\lceil\frac{d}{b}\rceil$ and $\rho(H)\leq\rho(H')$ for every component $H'$ of $G-S$ with $e(S,H')<\lceil\frac{d}{b}\rceil.$
By Lemma \ref{le2}, we have
\begin{flalign}\label{eqd}
&&\lambda_2(G)\geq\lambda_2(H_1\cup H_2)\geq {\rm min}\{\rho(H_1), \rho(H_2)\}\geq\rho(H).&&
\end{flalign}
By assumption, we know that $\lambda_2(G)<\phi(d,b).$ Combining (\ref{eqd}), we have $\rho(H)<\phi(d,b).$ For convenience, let $c=\lceil\frac{d}{b}\rceil.$

\vspace{1.5mm}
\noindent\textbf{Case 1.} $c\leq2$.
\vspace{1mm}

By Claim \ref{claim0} and Claim \ref{claim0.1}, we have $c=2$ and $d$ is odd.
Then $e(S,H)<c=2$. Since $G$ is connected, $e(S,H)=1.$  Then we have
\begin{eqnarray}\label{eq0}
n_H(n_H-1)\geq2m_H=dn_H-e(S,H)=dn_H-1.
\end{eqnarray}
We claim that $n_H\geq 2.$ In fact, if $n_H=1,$ then $e(S,H)=d\geq c=2,$ a contradiction. Hence $n_H\geq d+\frac{d-1}{n_H-1}>d,$ that is, $n_H\geq d+1.$ Note that $2m_H=dn_H-1$ and $d$ is odd. Then $n_H$ is odd, which implies that $n_H\geq d+2.$ Suppose that $n_H>d+2,$ that is,
$n_H\geq d+4.$ Note that $d\geq2$ and $d$ is odd. Then $d\geq3.$
By (\ref{eq0}) and Lemma \ref{le5}, we have
$$\rho(H)\geq\frac{2m_H}{n_H}=\frac{dn_H-1}{n_H}\geq d-\frac{1}{d+4}>\alpha_d=\phi(d,b),$$
a contradiction. Hence $n_H=d+2$ and $2m_H=d(d+2)-1$.
Then there must be $d+1$ vertices of degree $d$ and one vertex of degree $d-1$ in $H$, which implies that $H\cong(K_1\cup K_2)\vee\overline{\frac{d-1}{2}K_2}$. By (\ref{eqd}) and Lemma \ref{le4}, we have $$\lambda_2(G)\geq\rho(H)=\rho((K_1\cup K_2)\vee\overline{\frac{d-1}{2}K_2})=\alpha_d=\phi(d,b).$$ This contradicts the assumption of Theorem \ref{thm3}.

\vspace{1.5mm}
\noindent\textbf{Case 2.} $c\geq3$ is odd.
\vspace{1mm}

\vspace{1.5mm}
\noindent\textbf{Case 2.1.} $d$ and $n_H$ have different parity.
\vspace{1mm}

By Lemma \ref{le6}, we have $n_H=d+1$ and $2m_H=d(d+1)-c+1$. We claim that there are at least $d-c+2$ vertices of degree $d$. In fact, if there are at most $d-c+1$ vertices of degree $d,$
then $$d(d+1)-c+1=2m_H\leq (d-c+1)d+c(d-1)=d(d+1)-c,$$ a contradiction.
Let $V_1$ be the set of vertices of degree $d$ with $|V_1|=d-c+2$ and let $V_2$ be the remaining vertices in $V(H)$. The quotient matrix of $A(H)$ on the partition $(V_1,V_2)$ is
$$R_1(A(H))=\left[
\begin{array}{cc}
d-c+1&c-1\\
d-c+2&c-3
\end{array}
\right]
$$
whose characteristic polynomial is $P_1(x)=x^2-(d-2)x+c-2d-1.$
Since the largest root of $P_1(x)$ equals $\frac{d-2+\sqrt{d^2+4d+8-4c}}{2}$, by (\ref{eqd}) and Lemma \ref{le1}, we have $$\lambda_2(G)\geq\rho(H)\geq \lambda_1(R_1(A(H)))=\frac{d-2+\sqrt{d^2+4d+8-4c}}{2}=\phi(d,b),$$ a contradiction.

\vspace{1.5mm}
\noindent\textbf{Case 2.2.}  $d$ and $n_H$ are odd.
\vspace{1mm}

By Lemma \ref{le6}, we have $n_H=d+2$ and $2m_H=d(d+2)-c+2$. Then there are at least $d-c+4$ vertices of degree $d$.
Let $V_1$ be the set of vertices of degree $d$ with $|V_1|=d-c+4$ and let $V_2$ be the remaining vertices in $V(H)$. Denote by $m_{12}$ the number of edges between $V_1$ and $V_2$. Note that $(d-c+4)(c-3)\leq m_{12}\leq(d-c+4)(c-2)$. The quotient matrix of $A(H)$ on the partition $(V_1,V_2)$ is

$$R_2(A(H))=\left[
\begin{array}{cc}
d-\frac{m_{12}}{d-c+4}&\frac{m_{12}}{d-c+4}\\
\frac{m_{12}}{c-2}&d-\frac{m_{12}}{c-2}-1
\end{array}
\right],
$$
where its characteristic polynomial is $$P_2(x)=(x-d+\frac{m_{12}}{d-c+4})(x-d+\frac{m_{12}}{c-2}+1)-\frac{m_{12}^2}{(d-c+4)(c-2)}.$$
Note that $(d-c+4)(c-3)\leq m_{12}\leq(d-c+4)(c-2)$. Then $m_{12}=(d-c+4)(c-2)-t,$ where $0\leq t\leq d-c+4$. Let $g(x)=x^2-(d-3)x+c-3d-2$. Then we have
\begin{eqnarray*}
P_2(x)&=&x^2-(c+\frac{t}{c-2}-5)x-(d+2-c+\frac{t}{d-c+4})x\\
&&+(d+2-c+\frac{t}{d-c+4})(c+\frac{t}{c-2}-5)-\frac{((d-c+4)(c-2)-t)^2}{(d-c+4)(c-2)}\\
&=&g(x)+\frac{t}{(d-c+4)(c-2)}[-(d+2)x+d^2+2d-c+2].
\end{eqnarray*}
Let $\theta_1=\frac{d-3+\sqrt{d^2+6d+17-4c}}{2}$ be the largest root of $g(x)=0.$ Recall that $d\geq b+1$ and $b\geq1$. Then
$$P_2(\theta_1)=\frac{t}{(d-c+4)(c-2)}[-(d+2)\theta_1+d^2+2d-c+2]\leq0.$$
Combining (\ref{eqd}) and Lemma \ref{le1}, we obtain that$$\lambda_2(G)\geq\rho(H)\geq \lambda_1(R_2(A(H)))\geq\theta_1>\frac{d-2+\sqrt{d^2+4d+8-4c}}{2}=\phi(d,b),$$
a contradiction.

\vspace{1.5mm}
\noindent\textbf{Case 2.3.}  $d$ and $n_H$ are even.
\vspace{1mm}

By Lemma \ref{le6}, we have $n_H=d+2$ and $2m_H=d(d+2)-c+1$. Then there are at least $d-c+3$ vertices of degree $d$. Let $V_1$ be the set of vertices of degree $d$ with $|V_1|=d-c+3$ and let $V_2$ be the remaining vertices in $V(H)$. Denote by $m_{12}$ the number of edges between $V_1$ and $V_2$. Note that $(d-c+3)(c-2)\leq m_{12}\leq(d-c+3)(c-1)$. The quotient matrix of $A(H)$ on the partition $(V_1,V_2)$ is

$$R_3(A(H))=\left[
\begin{array}{cc}
d-\frac{m_{12}}{d-c+3}&\frac{m_{12}}{d-c+3}\\
\frac{m_{12}}{c-1}&d-\frac{m_{12}}{c-1}-1
\end{array}
\right]
$$
whose characteristic polynomial is $$P_3(x)=(x-d+\frac{m_{12}}{d-c+3})(x-d+\frac{m_{12}}{c-1}+1)-\frac{m_{12}^2}{(d-c+3)(c-1)}.$$
Note that $(d-c+3)(c-2)\leq m_{12}\leq(d-c+3)(c-1)$. Then $m_{12}=(d-c+3)(c-1)-t,$ where $0\leq t\leq d-c+3$. Let $h(x)=x^2-(d-3)x+c-3d-1$. Then we have
\begin{eqnarray*}
P_3(x)&=&x^2-(c+\frac{t}{c-1}-4)x-(d+1-c+\frac{t}{d-c+3})x\\
&&+(d+1-c+\frac{t}{d-c+3})(c+\frac{t}{c-1}-4)-\frac{[(d-c+3)(c-1)-t]^2}{(d-c+3)(c-1)}\\
&=&h(x)+\frac{t}{(d-c+4)(c-2)}[-(d+2)x+d^2+2d-c+1].
\end{eqnarray*}
Let $\theta_2=\frac{d-3+\sqrt{d^2+6d+13-4c}}{2}$ be the largest root of $h(x)=0.$ Note that $d\geq b+1$ and $b\geq1$. Then
$$P_3(\theta_2)=\frac{t}{2(d-c+3)(c-1)}[-2(d+2)\theta_2+2d^2+4d-2c+2]\leq0.$$
Combining (\ref{eqd}) and Lemma \ref{le1}, we have $$\lambda_2(G)\geq\rho(H)\geq \lambda_1(R_3(A(H)))\geq\theta_2>\frac{d-2+\sqrt{d^2+4d+8-4c}}{2}=\phi(d,b),$$
a contradiction.

\vspace{1.5mm}
\noindent\textbf{Case 3.} $c\geq3$ is even and $d$ is odd.
\vspace{1mm}

\vspace{1.5mm}
\noindent\textbf{Case 3.1.}  $n_H$ is odd.
\vspace{1mm}

By Lemma \ref{le6}, we have $n_H=d+2$ and $2m_H=d(d+2)-c+1$.
Similar to the proof of Case $2.3$, we have
$$\lambda_2(G)\geq\frac{d-3+\sqrt{d^2+6d+13-4c}}{2}=\phi(d,b),$$ a contradiction.

\vspace{1.5mm}
\noindent\textbf{Case 3.2.}  $n_H$ is even.
\vspace{1mm}

By Lemma \ref{le6}, we have $n_H=d+1$ and $2m_H=d(d+1)-c+2$. Then there are at least $d-c+3$ vertices of degree $d$. Let $V_1$ be the set of vertices of degree $d$ with $|V_1|=d-c+3$ and let $V_2$ be the remaining vertices in $V(H)$. The quotient matrix of $A(H)$ on the partition $(V_1,V_2)$ is
$$R_4(A(H))=\left[
\begin{array}{cc}
d-c+2&c-2\\
d-c+3&c-4
\end{array}
\right],
$$
whose characteristic polynomial is $P_4(x)=x^2-(d-2)x+c-2d-2.$ Since the largest root of $P_4(x)$ equals $\theta_3=\frac{d-2+\sqrt{d^2+4d+12-4c}}{2}$, by (\ref{eqd}) and Lemma \ref{le1}, we have $$\lambda_2(G)\geq\rho(H)\geq\lambda_1(R_4(A(H)))=\theta_3>\frac{d-3+\sqrt{d^2+6d+13-4c}}{2}=\phi(d,b).$$ This contradicts the assumption of Theorem \ref{thm3}.

\vspace{1.5mm}
\noindent\textbf{Case 4.} Both $c\geq3$ and $d$ are even.
\vspace{1mm}

\vspace{1.5mm}
\noindent\textbf{Case 4.1.}  $n_H$ is odd.
\vspace{1mm}

By Lemma \ref{le6}, we have $n_H=d+1$ and $2m_H=d(d+1)-c+2$.
Similar to the proof of Case $3.2$, we have
$$\lambda_2(G)\geq\frac{d-2+\sqrt{d^2+4d+12-4c}}{2}=\phi(d,b),$$ a contradiction.

\vspace{1.5mm}
\noindent\textbf{Case 4.2.}  $n_H$ is even.
\vspace{1mm}

By Lemma \ref{le6}, we have $n_H=d+2$ and $2m_H=d(d+2)-c+2$.
Similar to the proof of Case $2.2$, we have
$$\lambda_2(G)\geq\theta_1=\frac{d-3+\sqrt{d^2+6d+17-4c}}{2}>\frac{d-2+\sqrt{d^2+4d+12-4c}}{2}=\phi(d,b),$$ a contradiction.
\hspace*{\fill}$\Box$

\section{Graphs implying best bounds}
The next lemmas show that the upper bound $\phi(d,b)$ in Theorem \ref{thm3} is best possible.
Define graph $H(d,b)$ as follows.
\begin{eqnarray*}
\begin{split}
H(d,b)= \left \{
\begin{array}{ll}
(K_1\cup K_2)\vee\overline{\frac{d-1}{2}K_2} & \mbox{if $\lceil\frac{d}{b}\rceil\leq2$,}\\
K_{d-\lceil\frac{d}{b}\rceil+2}\vee\overline{\frac{\lceil\frac{d}{b}\rceil-1}{2}K_2}  & \mbox{if $\lceil\frac{d}{b}\rceil\geq3$ is odd,}\\
\overline{C_{\lceil\frac{d}{b}\rceil-1}}\vee\overline{\frac{d-\lceil\frac{d}{b}\rceil+3}{2}K_2}
& \mbox{if $\lceil\frac{d}{b}\rceil\geq3$ is even and $d$ is odd},\\
K_{d-\lceil\frac{d}{b}\rceil+3}\vee\overline{\frac{\lceil\frac{d}{b}\rceil-2}{2}K_2}  & \mbox{if $\lceil\frac{d}{b}\rceil\geq3$ and $d$ are even.}
\end{array}
\right.
\end{split}
\end{eqnarray*}

Let $\lceil\frac{d}{b}\rceil\geq3$ be odd. Take $d$ disjoint copies of
$H(d,b)=K_{d-\lceil\frac{d}{b}\rceil+2}\vee\overline{\frac{\lceil\frac{d}{b}\rceil-1}{2}K_2},$
add a new vertex set $S$ of $\lceil\frac{d}{b}\rceil-1$ vertices, and match all vertices of $S$ to $\lceil\frac{d}{b}\rceil-1$ vertices of degree $d-1$ in each copy of $H(d,b)$. The constructed $n$-vertex graph is denoted by $G^{\star}_1(d,b)$ (see Fig. \ref{f1}).
\begin{figure}[H]
\centering
\includegraphics[width=0.5\textwidth]{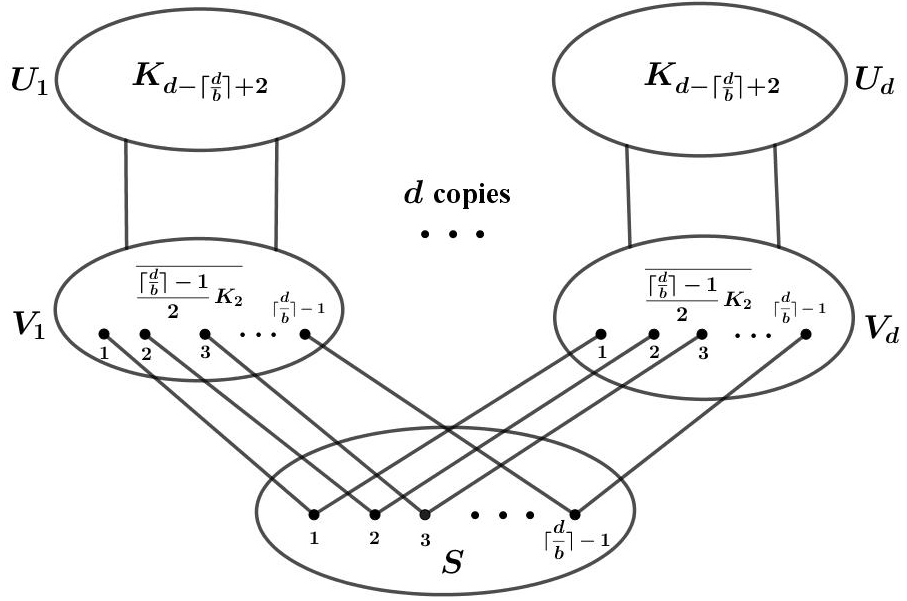}\\
\caption{Graph $G^{\star}_1(d,b).$
}\label{f1}
\end{figure}
\begin{lem}\label{le8}
For odd $\lceil\frac{d}{b}\rceil\geq3$,
$\lambda_2(G^{\star}_1(d,b))=\frac{d-2+\sqrt{d^2+4d+8-4\lceil\frac{d}{b}\rceil}}{2},$
yet graph $G^{\star}_1(d,b)$ is not $\frac{1}{b}$-tough.
\end{lem}

\begin{proof}
For each $1\leq i\leq d,$ let $U_i\cup V_i$ be a $2$-part equitable partition of
$K_{d-\lceil\frac{d}{b}\rceil+2}\vee\overline{\frac{\lceil\frac{d}{b}\rceil-1}{2}K_2}$ (see Fig. \ref{f1}). The quotient matrix of $A(H(d,b))$ on the vertex partition $(U_i, V_i)$ equals
$$\left[
\begin{array}{cc}
d-\lceil\frac{d}{b}\rceil+1&\lceil\frac{d}{b}\rceil-1\\
d-\lceil\frac{d}{b}\rceil+2&\lceil\frac{d}{b}\rceil-3
\end{array}
\right],
$$
whose eigenvalues are $\mu_1=\frac{d-2+\sqrt{d^2+4d+8-4\lceil\frac{d}{b}\rceil}}{2}$ and $\mu_2=\frac{d-2-\sqrt{d^2+4d+8-4\lceil\frac{d}{b}\rceil}}{2}$. Let $x$ and $y$ be eigenvectors of $H(d,b)$ corresponding to $\mu_1$ and $\mu_2$.
For each $1\leq i\leq d-1,$ construct $n$-dimensional vectors $\eta_1, \eta_2, \ldots, \eta_i, \ldots, \eta_{d-1}$ as follows:
$$\eta_1=\left[
\begin{array}{ccccccc}
0\\x\\-x\\0\\\vdots\\0\\0
\end{array}
\right],~\eta_2=\left[
\begin{array}{ccccccc}
0\\0\\x\\-x\\\vdots\\0\\0
\end{array}
\right], \ldots,~\eta_{d-1}=\left[
\begin{array}{ccccccc}
0\\0\\0\\0\\\vdots\\x\\-x
\end{array}
\right], $$
where $\eta_i|_{U_i\cup V_i}=x, \eta_i|_{U_{i+1}\cup V_{i+1}}=-x,$ and the remaining elements are all zero.
Meanwhile, construct $n$-dimensional vectors $\eta_{d}, \eta_{d+1}, \ldots, \eta_{d+i-1}, \ldots, \eta_{2d-2}$ as follows:
$$\eta_d=\left[
\begin{array}{ccccccc}
0\\y\\-y\\0\\\vdots\\0\\0
\end{array}
\right],~\eta_{d+1}=\left[
\begin{array}{ccccccc}
0\\0\\y\\-y\\\vdots\\0\\0
\end{array}
\right], \ldots,~\eta_{2d-2}=\left[
\begin{array}{ccccccc}
0\\0\\0\\0\\\vdots\\y\\-y
\end{array}
\right],$$
where $\eta_{d+i-1}|_{U_i\cup V_i}=y, \eta_{d+i-1}|_{U_{i+1}\cup V_{i+1}}=-y,$
and the remaining elements are all zero.
We always use $J$ to denote the all-one matrix, $I$ to denote the identity square matrix and $O$ to denote the zero matrix. The adjacency matrix $A(G^{\star}_1(d,b))$ on the partition $(S,
V_1\cup U_1, V_2\cup U_2, \ldots, V_d\cup U_d)$ is
$$A(G^{\star}_1(d,b))=\left[
\begin{array}{ccccc}
O&B&B&\cdots&B\\
B^T&A(H(d,b))&O&\cdots&O\\
B^T&O&A(H(d,b))&\cdots&O\\
\vdots&\vdots&\vdots&\ddots&\vdots\\
B^T&O&O&\cdots&A(H(d,b))
\end{array}
\right],
$$
where $B=[I,O]$ and
$$A(H(d,b))=\left[
\begin{array}{cc}
A(\overline{\frac{\lceil\frac{d}{b}\rceil-1}{2}K_2})&J\\
J^T&A(K_{d-\lceil\frac{d}{b}\rceil+2})
\end{array}
\right].
$$
One can check that $A(G^{\star}_1(d,b))\eta_i=\mu_1\eta_i$ and $A(G^{\star}_1(d,b))\eta_{d+i-1}=\mu_2\eta_{d+i-1}$, where $1\leq i\leq d-1.$
Hence $\mu_1$ and $\mu_2$ are also eigenvalues of $G^{\star}_1(d,b).$

The quotient matrix of $A(G^{\star}_1(d,b))$ on the $3$-part equitable partition $(S, V_1\cup V_2\cup\cdots\cup V_d, U_1\cup U_2\cup\cdots\cup U_d)$ equals
$$\left[
\begin{array}{ccc}
0&d&0\\
1&\lceil\frac{d}{b}\rceil-3&d-\lceil\frac{d}{b}\rceil+2\\
0&\lceil\frac{d}{b}\rceil-1&d-\lceil\frac{d}{b}\rceil+1
\end{array}
\right],
$$
whose eigenvalues are $d$ and $-1\pm\sqrt{d-\lceil\frac{d}{b}\rceil+2}.$
By Lemma \ref{le1}, they are also eigenvalues of $G^{\star}_1(d,b).$
Let $\eta_{2d-1}=\mathbf{1}$, $\eta_{2d}$ and $\eta_{2d+1}$ be eigenvectors corresponding to the above three eigenvalues. It is obvious that these three eigenvalues are different from $\mu_1$ and $\mu_2$. Moreover, $-1+\sqrt{d-\lceil\frac{d}{b}\rceil+2}<\mu_1=\frac{d-2+\sqrt{d^2+4d+8-4\lceil\frac{d}{b}\rceil}}{2}.$

Define $W=\mbox{span}\{\eta_1, \eta_2, \ldots, \eta_{2d+1}\}$. Let $\lambda$ be arbitrary one of the remaining eigenvalues. We can choose its eigenvector $\eta$ such that $\eta\bot W$.
Note that $W=\mbox{span}\{\xi_1, \xi_2, \xi_3, \ldots, \xi_{2d}, \xi_{2d+1}\}$, where $\xi_1|_S=\mathbf{1}$ and $\xi_1|_{\overline{S}}=\mathbf{0},$ $\xi_2|_{V_1}=\mathbf{1}$ and $\xi_2|_{\overline{V_1}}=\mathbf{0},$
$\xi_3|_{U_1}=\mathbf{1}$ and $\xi_3|_{\overline{U_1}}=\mathbf{0}, \ldots,$
$\xi_{2d}|_{V_d}=\mathbf{1}$ and $\xi_{2d}|_{\overline{V_d}}=\mathbf{0},$
$\xi_{2d+1}|_{U_d}=\mathbf{1}$ and $\xi_{2d+1}|_{\overline{U_d}}=\mathbf{0}.$
Hence $\eta^T\xi_i=0$ for $1\leq i\leq2d+1,$
so we have
\begin{eqnarray}\label{eqdd1}
J\cdot \eta|_S=\mathbf{0}, J\cdot \eta|_{V_1}=\mathbf{0}, J\cdot \eta|_{U_1}=\mathbf{0}, \ldots, J\cdot \eta|_{V_d}=\mathbf{0}, J\cdot \eta|_{U_d}=\mathbf{0}.
\end{eqnarray}

The adjacency matrix $A(G^{\star}_1(d,b))$ on the partition $(S, V_1, U_1, V_2, U_2, \ldots, V_d, U_d)$ is
$$A(G^{\star}_1(d,b))=\left[
\begin{array}{cccccc}
O&I&O&\cdots&I&O\\
I&A(\overline{\frac{\lceil\frac{d}{b}\rceil-1}{2}K_2})&J&\cdots&O&O\\
O&J^T&A(K_{d-\lceil\frac{d}{b}\rceil+2})&\cdots&O&O\\
\vdots&\vdots&\vdots&\ddots&\vdots&\vdots\\
I&O&O&\cdots&A(\overline{\frac{\lceil\frac{d}{b}\rceil-1}{2}K_2})&J\\
O&O&O&\cdots&J^T&A(K_{d-\lceil\frac{d}{b}\rceil+2})
\end{array}
\right].
$$
Let $G_1'(d,b)$ be the graph obtained from $G^{\star}_1(d,b)$ by removing all edges between $U'=\cup^d_{i=1}U_i$ and $V'=\cup^d_{i=1}V_i$. Then the adjacency matrix of graph $G_1'(d,b)$ is
$$A(G_1'(d,b))=\left[
\begin{array}{cccccc}
O&I&O&\cdots&I&O\\
I&A(\overline{\frac{\lceil\frac{d}{b}\rceil-1}{2}K_2})&O&\cdots&O&O\\
O&O&A(K_{d-\lceil\frac{d}{b}\rceil+2})&\cdots&O&O\\
\vdots&\vdots&\vdots&\ddots&\vdots&\vdots\\
I&O&O&\cdots&A(\overline{\frac{\lceil\frac{d}{b}\rceil-1}{2}K_2})&O\\
O&O&O&\cdots&O&A(K_{d-\lceil\frac{d}{b}\rceil+2})
\end{array}
\right].
$$
By (\ref{eqdd1}), we have
\begin{eqnarray*}
\lambda\eta=A(G^{\star}_1(d,b))\eta=A(G_1'(d,b))\eta.
\end{eqnarray*}
Hence $\lambda$ is also an eigenvalue of $G_1'(d,b).$
The quotient matrix of $A(G'_1(d,b))$ on the $3$-part equitable partition $(S, V', U')$ is
$$\left[
\begin{array}{ccc}
0&d&0\\
1&\lceil\frac{d}{b}\rceil-3&0\\
0&0&d-\lceil\frac{d}{b}\rceil+1
\end{array}
\right],
$$
whose eigenvalues are $d-\lceil\frac{d}{b}\rceil+1$ and $\frac{\lceil\frac{d}{b}\rceil-3\pm\sqrt{{\lceil\frac{d}{b}\rceil}^2-6\lceil\frac{d}{b}\rceil+4d+9}}{2}.$ By Lemma \ref{le1}, we have
\begin{eqnarray*}
\lambda\leq\rho(G_1'(d,b))&=&\mathrm{max}\left\{d-\lceil\frac{d}{b}\rceil+1, \frac{\lceil\frac{d}{b}\rceil-3\pm\sqrt{{\lceil\frac{d}{b}\rceil}^2-6\lceil\frac{d}{b}\rceil+4d+9}}{2}\right\}
\\&<&\mu_1=\frac{d-2+\sqrt{d^2+4d+8-4\lceil\frac{d}{b}\rceil}}{2}.
\end{eqnarray*}
Hence $$\lambda_2(G^{\star}_1(d,b))=\frac{d-2+\sqrt{d^2+4d+8-4\lceil\frac{d}{b}\rceil}}{2}.$$

Next we prove that $G^{\star}_1(d,b)$ is not $\frac{1}{b}$-tough.
In fact,
$$\tau(G^{\star}_1(d,b))\leq\frac{|S|}{c(G^{\star}_1(d,b)-S)}=\frac{\lceil\frac{d}{b}\rceil-1}{d}<\frac{1}{b},$$ and hence $G^{\star}_1(d,b)$ is not $\frac{1}{b}$-tough.
\end{proof}

Using the techniques of Lemma \ref{le8}, we can prove the following Lemmas \ref{le9}, \ref{le10} and \ref{le11}. These lemmas imply that the upper bound $\phi(d,b)$ in Theorem \ref{thm3} is best possible.

Let $\lceil\frac{d}{b}\rceil\geq3$ be even and let $d\geq1$ be odd.
Take $d$ disjoint copies of $\overline{C_{\lceil\frac{d}{b}\rceil-1}}\vee\overline{\frac{d-\lceil\frac{d}{b}\rceil+3}{2}K_2},$
add a new vertex set $S$ of $\lceil\frac{d}{b}\rceil-1$ vertices, and match all vertices of $S$ to $\lceil\frac{d}{b}\rceil-1$ vertices of degree $d-1$ in each copy of $H(d,b)$. The constructed $n$-vertex graph is denoted by $G^{\star}_2(d,b)$ (see Fig. \ref{f2}).

\begin{figure}[H]
\centering
\includegraphics[width=0.5\textwidth]{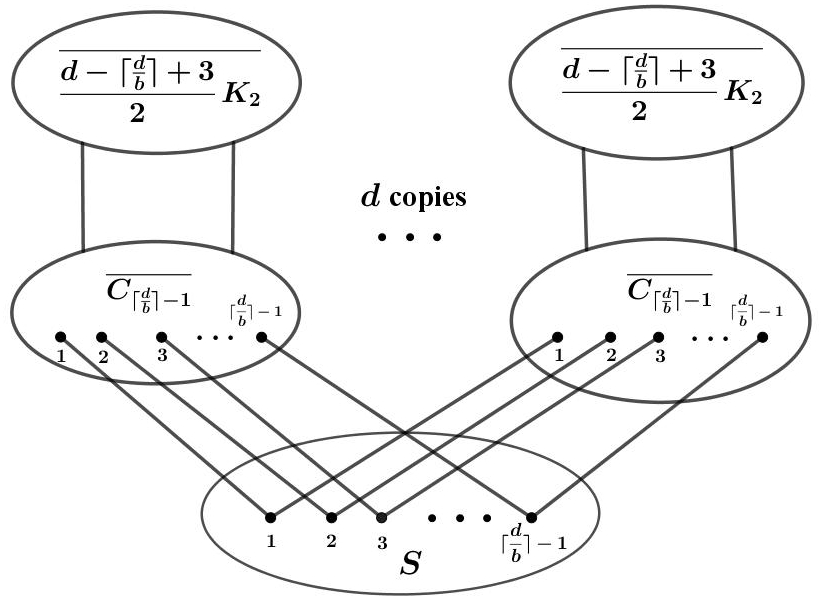}\\
\caption{Graph $G^{\star}_2(d,b).$
}\label{f2}
\end{figure}

\begin{lem}\label{le9}
For even $\lceil\frac{d}{b}\rceil\geq3$ and odd $d\geq1$,
$\lambda_2(G^{\star}_2(d,b))=\frac{d-3+\sqrt{d^2+6d+13-4\lceil\frac{d}{b}\rceil}}{2},$
yet graph $G^{\star}_2(d,b)$ is not $\frac{1}{b}$-tough.
\end{lem}

Let $\lceil\frac{d}{b}\rceil\geq3$ and $d\geq1$ be even.
Take $d$ disjoint copies of $K_{d-\lceil\frac{d}{b}\rceil+3}\vee\overline{\frac{\lceil\frac{d}{b}\rceil-2}{2}K_2},$
add a new vertex set $S$ of $\lceil\frac{d}{b}\rceil-2$ vertices, and match all vertices of $S$ to $\lceil\frac{d}{b}\rceil-2$ vertices of degree $d-1$ in each copy of $H(d,b)$.  The constructed $n$-vertex graph is denoted by $G^{\star}_3(d,b)$ (see Fig. \ref{f3}).

\begin{figure}[H]
\centering
\includegraphics[width=0.5\textwidth]{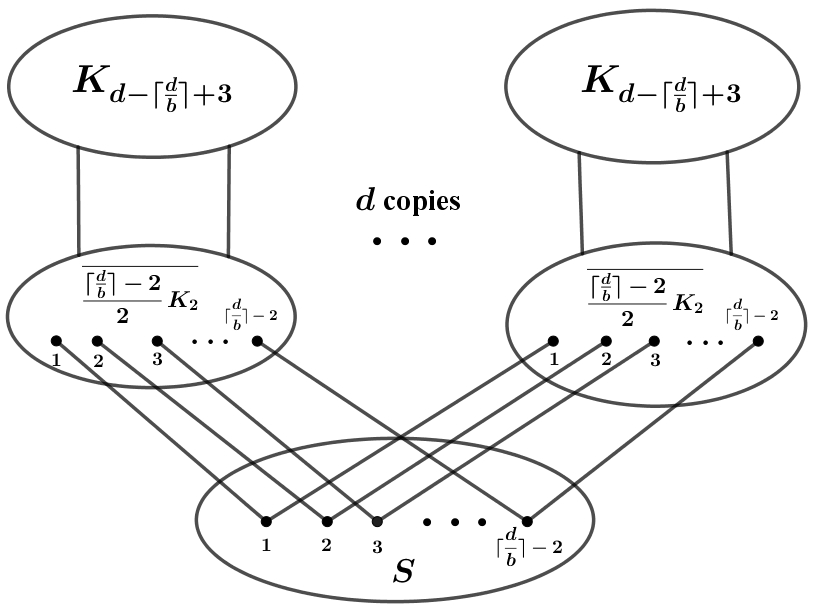}\\
\caption{Graph $G^{\star}_3(d,b).$
}\label{f3}
\end{figure}

\begin{lem}\label{le10}
For even $\lceil\frac{d}{b}\rceil\geq3$ and even $d\geq1$,
$\lambda_2(G^{\star}_3(d,b))=\frac{d-2+\sqrt{d^2+4d+12-4\lceil\frac{d}{b}\rceil}}{2},$
yet graph $G^{\star}_3(d,b)$ is not $\frac{1}{b}$-tough.
\end{lem}

Let $\lceil\frac{d}{b}\rceil=2$ and let $d\geq1$ be an odd integer.
Consider $d$ pairwise vertex disjoint copies of $(K_1\cup K_2)\vee\overline{\frac{d-1}{2}K_2}$. Let graph $G^{\star}_4(d,b)$ be obtained by adding $d$ edges between a
vertex set $S=\{u\}$ and a vertex of degree $d-1$ in each of the $d$ copies of $(K_1\cup K_2)\vee\overline{\frac{d-1}{2}K_2}$ (see Fig. \ref{f4}).

\begin{figure}[H]
\centering
\includegraphics[width=0.45\textwidth]{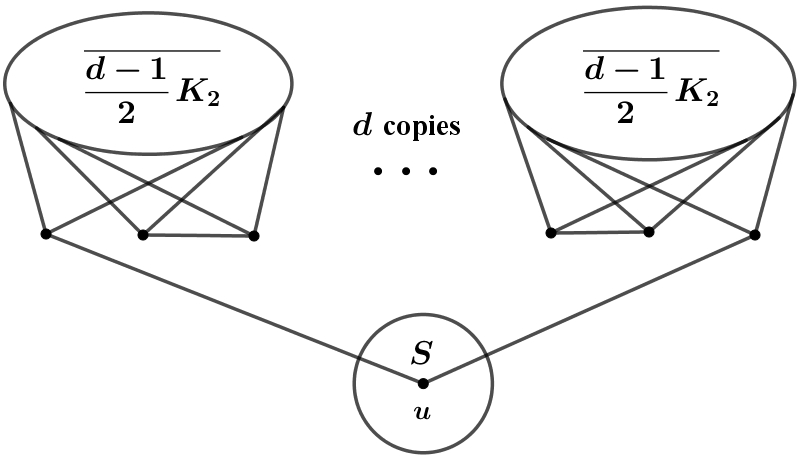}\\
\caption{Graph $G^{\star}_4(d,b).$
}\label{f4}\end{figure}

\begin{lem}\label{le11}
For $\lceil\frac{d}{b}\rceil=2$ and odd integer $d\geq1$,
$\lambda_2(G^{\star}_4(d,b))=\alpha_d,$
yet graph $G^{\star}_4(d,b)$ is not $\frac{1}{b}$-tough.
\end{lem}

\section{Proof of Theorem \ref{thm4}}
Before presenting the proof, we first introduce a necessary lemma.
\begin{lem}\label{le7}
Let $G$ be a connected $d$-regular graph and $b\geq1$ be an integer with $d\geq b+2.$
Let $H$ be a component of $G-S$ such that $e(S,H)\leq d-b$, where $S\subseteq V(G)$ is not an empty set. If $\rho(H)\leq\rho(H')$ for every component $H'$ of $G-S$ with $e(S,H')\leq d-b$ and $\rho(H)\leq\psi(d,b)$, then we have
\begin{eqnarray*}
\begin{split}
n_H= \left \{
\begin{array}{ll}
d+2  & \mbox{if $d$ and $n_H$ have the same parity,}\\
d+1  & \mbox{otherwise,}
\end{array}
\right.
\end{split}
\end{eqnarray*}
and
\begin{eqnarray*}
\begin{split}
2m_H= \left \{
\begin{array}{ll}
d(d+2)-d+b+1&\mbox{if $d$, $n_H$ have the same parity and $b$ is odd,}\\
d(d+2)-d+b & \mbox{if $d$, $n_H$ have the same parity and $b$ is even,}\\
d(d+1)-d+b+1 & \mbox{if $n_H$, $b$ are odd and $d$ is even or $n_H$, $b$ are even and $d$ is odd},\\
d(d+1)-d+b  & \mbox{if $d$, $b$ are odd and $n_H$ is even or $d$, $b$  are even and $n_H$ is odd.}
\end{array}
\right.
\end{split}
\end{eqnarray*}
\end{lem}

\begin{proof}
Note that $e(S,H)\leq d-b$ and $G$ is connected $d$-regular graph. Then we have
\begin{eqnarray}\label{eq10}
n_H(n_H-1)\geq2m_H=dn_H-e(S,H)\geq dn_H-d+b.
\end{eqnarray}
We claim that $n_H\geq 2.$ In fact, if $n_H=1,$ then $e(S,H)=d>d-b,$ a contradiction.
Hence $n_H\geq d+\frac{b}{n_H-1}>d,$ that is, $n_H\geq d+1.$ This implies that
\begin{eqnarray*}
\begin{split}
n_H\geq \left \{
\begin{array}{ll}
d+2  & \mbox{if $d$ and $n_H$ have the same parity,}\\
d+1  & \mbox{otherwise.}
\end{array}
\right.
\end{split}
\end{eqnarray*}
Assume that
\begin{eqnarray*}
\begin{split}
n_H> \left \{
\begin{array}{ll}
d+2  & \mbox{if $d$ and $n_H$ have the same parity,}\\
d+1  & \mbox{otherwise.}
\end{array}
\right.
\end{split}
\end{eqnarray*}

Assume that $d$ and $b$ have the same parity. Recall that $\psi(d,b)=\frac{d-2+\sqrt{d^2+4b+4}}{2}.$
If $d$, $b$ and $n_H$ are odd, then $n_H\geq d+4.$ By (\ref{eq10}), we have $2m_H\geq dn_H-d+b+1$. According to Lemma \ref{le3}, we have
$\rho(H)\geq\frac{2m_H}{n_H}\geq d-\frac{d-b-1}{d+4}
>\frac{d-2+\sqrt{d^2+4b+4}}{2}.$
If $d$, $b$ are odd and $n_H$ is even or $d$, $b$ are even and $n_H$ is odd, then $n_H\geq d+3$ and $2m_H\geq dn_H-d+b$. So we have $\rho(H)\geq\frac{2m_H}{n_H}\geq d-\frac{d-b}{d+3}>\frac{d-2+\sqrt{d^2+4b+4}}{2}.$
If $d$, $b$ and $n_H$ are even, then $n_H\geq d+4$ and $2m_H\geq dn_H-d+b$. Hence we obtain that $\rho(H)\geq\frac{2m_H}{n_H}\geq d-\frac{d-b}{d+4}>\frac{d-2+\sqrt{d^2+4b+4}}{2}.$

Suppose that $d$ is odd and $b$ is even. Then $\psi(d,b)=\frac{d-3+\sqrt{d^2+4b+2d+9}}{2}.$
If $d$, $n_H$ are odd and $b$ is even, then $n_H\geq d+4$ and $2m_H\geq dn_H-d+b$. So we have $\rho(H)\geq\frac{2m_H}{n_H}\geq d-\frac{d-b}{d+4}>\frac{d-3+\sqrt{d^2+4b+2d+9}}{2}.$
If $d$ is odd and $b$, $n_H$ are even, then $n_H\geq d+3$
and $2m_H\geq dn_H-d+b+1$. Hence we have $\rho(H)\geq\frac{2m_H}{n_H}\geq d-\frac{d-b-1}{d+3}>\frac{d-3+\sqrt{d^2+4b+2d+9}}{2}.$

Assume that $d$ is even and $b$ is odd. Recall that $\psi(d,b)=\frac{d-2+\sqrt{d^2+4b+8}}{2}.$
If $d$ is even and $b$, $n_H$ are odd, then $n_H\geq d+3$ and $2m_H\geq dn_H-d+b+1$. Therefore we have $\rho(H)\geq\frac{2m_H}{n_H}\geq d-\frac{d-b-1}{d+3}
>\frac{d-2+\sqrt{d^2+4b+8}}{2}.$
If $d$, $n_H$ are even and $b$ is odd, then $n_H\geq d+4$ and $2m_H\geq dn_H-d+b+1$. So we have
$\rho(H)\geq\frac{2m_H}{n_H}\geq d-\frac{d-b-1}{d+4}>\frac{d-2+\sqrt{d^2+4b+8}}{2}.$

In the above cases, we always have $\rho(H)>\psi(d,b),$ a contradiction. Then
\begin{eqnarray*}
\begin{split}
n_H= \left \{
\begin{array}{ll}
d+2  & \mbox{if $d$ and $n_H$ have the same parity,}\\
d+1  & \mbox{otherwise.}
\end{array}
\right.
\end{split}
\end{eqnarray*}

By (\ref{eq10}), we know that $2m_H\geq dn_H-d+b.$
Suppose that $2m_H>dn_H-d+b,$ that is,
\begin{eqnarray*}
\begin{split}
2m_H> \left \{
\begin{array}{ll}
d(d+2)-d+b+1 & \mbox{if $d$, $n_H$ have the same parity and $b$ is odd,}\\
d(d+2)-d+b & \mbox{if $d$, $n_H$ have the same parity and $b$ is even,}\\
d(d+1)-d+b+1 & \mbox{if $n_H$, $b$ are odd and $d$ is even or $n_H$, $b$ are even and $d$ is odd},\\
d(d+1)-d+b  & \mbox{if $d$, $b$ are odd and $n_H$ is even or $d$, $b$  are even and $n_H$ is odd.}
\end{array}
\right.
\end{split}
\end{eqnarray*}

Assume that $d$ and $b$ have the same parity. Recall that $\psi(d,b)=\frac{d-2+\sqrt{d^2+4b+4}}{2}.$
If $d$, $b$ and $n_H$ are odd, then $n_H=d+2$ and $2m_H\geq d(d+2)-d+b+3$. Hence we have
$\rho(H)\geq\frac{2m_H}{n_H}\geq d-\frac{d-b-3}{d+2}
>\frac{d-2+\sqrt{d^2+4b+4}}{2}.$
If $d$, $b$ are odd and $n_H$ is even or $d$, $b$ are even and $n_H$ is odd, then $n_H=d+1$ and $2m_H\geq d(d+1)-d+b+2$. So we have $\rho(H)\geq\frac{2m_H}{n_H}\geq d-\frac{d-b-2}{d+1}>\frac{d-2+\sqrt{d^2+4b+4}}{2}.$
If $d$, $b$ and $n_H$ are even, then $n_H=d+2$ and $2m_H\geq d(d+2)-d+b+2$. Hence we obtain that $\rho(H)\geq\frac{2m_H}{n_H}\geq d-\frac{d-b-2}{d+2}>\frac{d-2+\sqrt{d^2+4b+4}}{2}.$

Suppose that $d$ is odd and $b$ is even. Then $\psi(d,b)=\frac{d-3+\sqrt{d^2+4b+2d+9}}{2}.$
If $d$, $n_H$ are odd and $b$ is even, then $n_H=d+2$ and $2m_H\geq d(d+2)-d+b+2$. So we have $\rho(H)\geq\frac{2m_H}{n_H}\geq d-\frac{d-b-2}{d+2}>\frac{d-3+\sqrt{d^2+4b+2d+9}}{2}.$
If $d$ is odd and $b$, $n_H$ are even, then $n_H=d+1$
and $2m_H\geq d(d+1)-d+b+3$. Hence we have $\rho(H)\geq\frac{2m_H}{n_H}\geq d-\frac{d-b-3}{d+1}>\frac{d-3+\sqrt{d^2+4b+2d+9}}{2}.$

Assume that $d$ is even and $b$ is odd. Recall that $\psi(d,b)=\frac{d-2+\sqrt{d^2+4b+8}}{2}.$
If $d$ is even and $b$, $n_H$ are odd, then $n_H=d+1$ and $2m_H\geq d(d+1)-d+b+3$. Therefore we have $\rho(H)\geq\frac{2m_H}{n_H}\geq d-\frac{d-b-3}{d+1}
>\frac{d-2+\sqrt{d^2+4b+8}}{2}.$
If $d$, $n_H$ are even and $b$ is odd, then $n_H=d+2$
and $2m_H\geq d(d+2)-d+b+3$. So we have
$\rho(H)\geq\frac{2m_H}{n_H}\geq d-\frac{d-b-3}{d+2}>\frac{d-2+\sqrt{d^2+4b+8}}{2}.$

These above cases always contradict $\rho(H)\leq\psi(d,b)$. So we complete the proof.
\end{proof}

Now we are in a position to give the proof of Theorem \ref{thm4}.

\medskip
\noindent  \textbf{Proof of Theorem \ref{thm4}.}
Suppose that a connected $d$-regular graph $G$ not a $\frac{1}{b}$-tough graph. There exists some subset $S\subseteq V(G)$ such that $c(G-S)\geq b|S|+1$. We choose $|S|$ to be as small as possible. According to the definition of toughness, we know that $S$ is not an empty set. Hence $|S|\geq1$. Let $|S|=s$ and $c(G-S)=q$. Then $q\geq bs+1$. Let $H_1, H_2, \ldots, H_q$ be the components of $G-S,$ and let $e(S,H_i)$ be the number of edges in $G$ between $S$ and $H_i$. Note that $G$ is connected. It is obvious that $e(S,H_i)\geq1$ and $q\leq\sum_{i=1}^{q}e(S,H_i)\leq sd.$

\begin{claim}\label{claim11}
$d\geq b+1$.
\end{claim}
\begin{proof}
Every vertex in $S$ must be adjacent to at least one connected component in $\{ H_1, H_2, \ldots, H_q \}$. Otherwise, there exists some vertex $ u\in S$ such that all neighbors of $u$ are contained in $S$. Consider the subset $S'= S\setminus u$, then $c(G-S')=q+1$, which contradicts the minimality of $S$. Recall that $q\geq bs+1$ and $s\geq1$.
Consequently, there exists a vertex $v\in S$ satisfies
$$d\geq d_{G-S}(v)\geq\frac{q}{s}\geq\frac{bs+1}{s}>b,$$
and hence $d\geq b+1.$
\end{proof}

\begin{claim}\label{claim12}
If $d=b+1$, then $d$ is odd.
\end{claim}
\begin{proof}
Suppose that $d$ is even. Recall that there exists some subset $S\subseteq V(G)$ such that $q\geq bs+1$. Clearly, $2|E(H_i)|=d|V(H_i)|-e(S,H_i)$ for $i\in[1,q].$ Since $d$ is even, $e(S,H_i)\geq1$ is even. By $d=b+1\geq2$ and $b=d-1$, we have $$sd\geq\sum_{i=1}^{q}e(S,H_i)\geq2q\geq sd+(d-2)s+2 > sd,$$ a contradiction.
\end{proof}
\begin{claim}\label{claim1}
There are at least $b+1$ components, says $H_1$, $H_2, \ldots, H_{b+1}$ such that $e(S,H_i)\leq d-b$ for all $i\in\{1,2, \ldots, b+1\}$.
\end{claim}
\begin{proof}
If $s=1$, then  $e(S,H_i)\leq d-q+1\leq d-b$.
Next we consider $s\geq2.$
Assume to the contrary that there are at most $b$ such components in $G-S$. Since $G$ is $d$-regular, $q\geq bs+1$ and $d\geq b+1$, we have
$$sd>\sum_{i=b+1}^{q}e(S,H_i)\geq(q-b)(d-b+1)\geq sd+(b-1)((s-1)(d-b)-1)\geq sd+(b-1)(s-2)\geq sd,$$
which is a contradiction.
\end{proof}

Assume that $H$ is a component of $G-S$ such that $e(S,H)\leq d-b$ and $\rho(H)\leq\rho(H')$ for every component $H'$ of $G-S$ with $e(S,H')\leq d-b.$
By Lemma \ref{le3}, we have
\begin{flalign}\label{eqd1}
&&\lambda_{b+1}(G)\geq\lambda_{b+1}(H_1\cup H_2\cup\cdots\cup H_{b+1})\geq {\rm min}\{\rho(H_1), \rho(H_2), \ldots, \rho(H_{b+1})\}\geq\rho(H).&&
\end{flalign}
By assumption, we know that $\lambda_{b+1}(G)<\psi(d,b).$ Combining (\ref{eqd1}), we have $\rho(H)<\psi(d,b).$

\vspace{1.5mm}
\noindent\textbf{Case 1.} $d\leq b+1$.
\vspace{1mm}

By Claim \ref{claim11} and Claim \ref{claim12}, we have $d=b+1$ and $d$ is odd.
Then $e(S,H)\leq d-b=1$. Since $G$ is connected, $e(S,H)=1.$  Then we have
\begin{eqnarray}\label{eq000}
n_H(n_H-1)\geq2m_H=dn_H-e(S,H)=dn_H-1.
\end{eqnarray}
We claim that $n_H\geq 2.$ In fact, if $n_H=1,$ then $e(S,H)=d>d-b=1,$ a contradiction. Hence $n_H\geq d+\frac{d-1}{n_H-1}>d,$ that is, $n_H\geq d+1.$ Note that $2m_H=dn_H-1$ and $d$ is odd. Then $n_H$ is odd, which implies that $n_H\geq d+2.$ Suppose that $n_H>d+2,$ that is,
$n_H\geq d+4.$ Note that $d\geq2$ and $d$ is odd. Then $d\geq3.$
By (\ref{eq000}) and Lemma \ref{le5}, we have
$$\rho(H)\geq\frac{2m_H}{n_H}=\frac{dn_H-1}{n_H}\geq d-\frac{1}{d+4}>\alpha_d=\psi(d,b),$$
a contradiction. Hence $n_H=d+2$ and $2m_H=d(d+2)-1$.
Then there must be $d+1$ vertices of degree $d$ and one vertex of degree $d-1$ in $H$, which implies that $H\cong(K_1\cup K_2)\vee\overline{\frac{d-1}{2}K_2}$. By Lemma (\ref{eqd1}) and \ref{le4}, we have $$\lambda_2(G)\geq\rho(H)\geq\rho((K_1\cup K_2)\vee\overline{\frac{d-1}{2}K_2})=\alpha_d=\psi(d,b).$$ This contradicts the assumption of Theorem \ref{thm4}.

\vspace{1.5mm}
\noindent\textbf{Case 2.} $d\geq b+2$ and $b$ have same parity.
\vspace{1mm}

\vspace{1.5mm}
\noindent\textbf{Case 2.1.} $d$, $b$ are even and $n_H$ is odd or $d$, $b$ are odd and $n_H$ is even.
\vspace{1mm}

By Lemma \ref{le7}, we have $n_H=d+1$ and $2m_H=d(d+1)-d+b$.
We claim that there are at least $b+1$ vertices of degree $d$. In fact, if there are at most $b$ vertices of degree $d$, then $$d(d+1)-d+b=2m_H\leq bd+(d-b+1)(d-1)=d(d+1)-d+b-1,$$ a contradiction.
Let $V_1$ be the set of vertices of degree $d$ with $|V_1|=b+1$ and let $V_2$ be the remaining vertices in $V(H)$. The quotient matrix of $A(H)$ on the partition $(V_1,V_2)$ is
$$R'_1(A(H))=\left[
\begin{array}{cc}
b&d-b\\
b+1&d-b-2
\end{array}
\right]
$$
whose characteristic polynomial is $P_5(x)=x^2-(d-2)x-b-d.$
Since the largest root of $P_5(x)$ equals $\frac{d-2+\sqrt{d^2+4b+4}}{2}$, by (\ref{eqd1}) and Lemma \ref{le1}, we have $$\lambda_{b+1}(G)\geq\rho(H)\geq \rho(R'_1(A(H)))=\frac{d-2+\sqrt{d^2+4b+4}}{2}=\psi(d,b),$$
a contradiction.

\vspace{1.5mm}
\noindent\textbf{Case 2.2.}  $d$, $b$ and $n_H$ are odd.
\vspace{1mm}

By Lemma \ref{le7}, we have $n_H=d+2$ and $2m_H=d(d+2)-d+b+1$. Then there are at least $d+3$ vertices of degree $d$. Let $V_1$ be the set of vertices of degree $d$ with $|V_1|=d+3$ and let $V_2$ be the remaining vertices in $V(H)$. Denote by $m^*_{12}$ the number of edges between $V_1$ and $V_2$. Note that $(d-b-2)(b+3)\leq m^*_{12}\leq(d-b-1)(b+3)$. The quotient matrix of $A(H)$ on the partition $(V_1,V_2)$ is
$$R'_2(A(H))=\left[
\begin{array}{cc}
d-\frac{m^*_{12}}{b+3}&\frac{m^*_{12}}{b+3}\\
\frac{m^*_{12}}{d-b-1}&d-\frac{m^*_{12}}{d-b-1}-1
\end{array}
\right]
$$
whose characteristic polynomial is $$P_6(x)=(x-d+\frac{m^*_{12}}{b+3})(x-d+\frac{m^*_{12}}{d-b-1}+1)-\frac{{m^*_{12}}^2}{(d-b-1)(b+3)}.$$
Note that $(d-b-2)(b+3)\leq m^*_{12}\leq(d-b-1)(b+3)$. Then $m^*_{12}=(d-b-1)(b+3)-t,$ where $0\leq t\leq b+3$. Let $\varphi(x)=x^2-(d-3)x-b-2d-1$. Then we have
\begin{eqnarray*}
P_6(x)&=&x^2-(d-b-4+\frac{t}{d-b-1})x-(b+1+\frac{t}{b+3})x\\
&&+(b+1+\frac{t}{b+3})(d-b-4+\frac{t}{d-b-1})-\frac{((d-b-1)(b+3)-t)^2}{(d-b-1)(b+3)}\\
&=&\varphi(x)+\frac{t}{(d-b-1)(b+3)}[-(d+2)x+d^2+b+d+1].
\end{eqnarray*}
Let $\theta_4=\frac{d-3+\sqrt{d^2+4b+2d+13}}{2}$ be the largest root of  $\varphi(x)=0$. Recall that $d>b+1$ and $b\geq1$. Then
$$P_6(\theta_4)=\frac{t}{(d-b-1)(b+3)}[-2(d+2)\theta_4+d^2+b+d+1]\leq0.$$
Combining (\ref{eqd1}) and Lemma \ref{le1}, we obtain that$$\lambda_{b+1}(G)\geq\rho(H)\geq \lambda_1(R'_2(A(H)))\geq\theta_4>\frac{d-2+\sqrt{d^2+4b+4}}{2}=\psi(d,b),$$
a contradiction.

\vspace{1.5mm}
\vspace{1.5mm}
\noindent\textbf{Case 2.3.}  $d$, $b$ and $n_H$ are even.
\vspace{1mm}

By Lemma \ref{le7}, we have $n_H=d+2$ and $2m_H=d(d+2)-d+b$. Then there are at least $d+2$ vertices of degree $d$. Let $V_1$ be a set of vertices with degree $d$ such that $|V_1|=d+2$ and let $V_2$ be the remaining vertices in $V(H)$. Denote by $m^*_{12}$ the number of edges between $V_1$ and $V_2$. Note that $(b+2)(d-b-1)\leq m^*_{12}\leq(b+2)(d-b)$. Then the quotient matrix of $A(H)$ on the partition $(V_1, V_2)$ is

$$R'_3(A(H))=\left[
\begin{array}{cc}
d-\frac{m^*_{12}}{b+2}&\frac{m^*_{12}}{b+2}\\
\frac{m^*_{12}}{d-b}&d-\frac{m^*_{12}}{d-b}-1
\end{array}
\right]
$$
whose characteristic polynomial is $$P_7(x)=(x-d+\frac{m^*_{12}}{b+2})(x-d+\frac{m^*_{12}}{d-b}+1)-\frac{{m^*_{12}}^2}{(b+2)(d-b)}.$$
Note that $(b+2)(d-b-1)\leq m_{12}\leq(b+2)(d-b)$. Then $m^*_{12}=(b+2)(d-b)-t,$ where $0\leq t\leq b+2$. Let $\omega(x)=x^2-(d-3)x-b-2d$. Then
\begin{eqnarray*}
P_7(x)&=&x^2-(d-b-5+\frac{t}{d-b-2})x-(b+\frac{t}{b+2})x\\
&&+(b+\frac{t}{b+2})(d-b-5+\frac{t}{d-b-2})-\frac{((b+2)(d-b)-t)^2}{(b+2)(d-b)}\\
&=&\omega(x)+\frac{t}{(b+2)(d-b)}[-(d+2)x+d^2+b+d].
\end{eqnarray*}
Let $\theta_5=\frac{d-3+\sqrt{d^2+4b+2d+9}}{2}$ be the largest root of  $\omega(x)=0$. Note that $d>b+1$ and $b\geq1$. Then
$$P_7(\theta_5)=\frac{t}{(b+2)(d-b)}[-2(d+2)\theta_5+d^2+b+d]\leq0.$$
Combining this with (\ref{eqd1}), we obtain that$$\lambda_{b+1}(G)\geq\rho(H)\geq \lambda_1(R'_3(A(H)))\geq\theta_5>\frac{d-2+\sqrt{d^2+4b+4}}{2}=\psi(d,b),$$
a contradiction.

\vspace{1.5mm}
\noindent\textbf{Case 3.} $d\geq b+2$ is odd and $b$ is even.
\vspace{1mm}

\vspace{1.5mm}
\noindent\textbf{Case 3.1.}  $n_H$ is odd.
\vspace{1mm}

By Lemma \ref{le7}, we have $n_H=d+2$ and $2m_H=d(d+2)-d+b$.
Similar to the proof of Case $2.3$, we have
$$\lambda_{b+1}(G)\geq\frac{d-3+\sqrt{d^2+4b+2d+9}}{2}=\psi(d,b),$$ a contradiction.

\vspace{1.5mm}
\noindent\textbf{Case 3.2.}  $n_H$ is even.
\vspace{1mm}

By Lemma \ref{le7}, we have $n_H=d+1$ and $2m_H=d(d+1)-d+b+1$. Then there are at least $b+2$ vertices of degree $d$. Let $V_1$ be a set of vertices with degree $d$ such that $|V_1|=b+2$ and let $V_2$ be the remaining vertices in $V(H)$. The quotient matrix of $A(H)$ on the partition $(V_1,V_2)$ is
$$R'_4(A(H))=\left[
\begin{array}{cc}
b+1&d-b-1\\
b+2&d-b-3
\end{array}
\right]
$$
whose characteristic polynomial is $P_8(x)=x^2-(d-2)x-b-d-1.$
Since the largest root of $P_8(x)$ equals $\theta_6=\frac{d-2+\sqrt{d^2+4b+8}}{2}$, by (\ref{eqd1}) and Lemma \ref{le1}, we have
$$\lambda_{b+1}(G)\geq\rho(H)\geq \rho(R'_4(A(H)))=\theta_6>\frac{d-3+\sqrt{d^2+4b+2d+9}}{2}=\psi(d,b),$$ a contradiction.

\vspace{1.5mm}
\noindent\textbf{Case 4.} $d\geq b+2$ is even and $b$ is odd.
\vspace{1mm}

\vspace{1.5mm}
\noindent\textbf{Case 4.1.}  $n_H$ is odd.
\vspace{1mm}

By Lemma \ref{le7}, we have $n_H=d+1$ and $2m_H=d(d+1)-d+b+1$.
Similar to the proof of Case $3.2$, we have
$$\lambda_{b+1}(G)\geq\frac{d-2+\sqrt{d^2+4b+8}}{2}=\psi(d,b),$$ a contradiction.

\vspace{1.5mm}
\noindent\textbf{Case 4.2.}  $n_H$ is even.
\vspace{1mm}

By Lemma \ref{le7}, we have $n_H=d+2$ and $2m_H=d(d+2)-d+b+1$.
Similar to the proof of Case $2.2$, we have
$$\lambda_{b+1}(G)\geq\frac{d-3+\sqrt{d^2+4b+2d+13}}{2}>\frac{d-2+\sqrt{d^2+4b+8}}{2}=\psi(d,b).$$ This contradicts the assumption $\lambda_{b+1}(G)<\psi(d,b)$ of Theorem \ref{thm3}.
\hspace*{\fill}$\Box$

\section{Concluding remarks}

In this paper, we prove a best possible upper bound $\phi(d,b)$ of $\lambda_2(G)$ for a connected $d$-regular graph to be $\frac{1}{b}$-tough, where $b$ is a positive integer. Furthermore, we also find an upper bound $\psi(d,b)$ of $\lambda_{b+1}(G)$ to ensure that a connected $d$-regular graph is $\frac{1}{b}$-tough.
Chen et al.\cite{Chen2025} indicated that the upper bound $\psi(d,b)$ is best possible only for $d\leq b+1$.
\begin{lem}[Chen et al.\cite{Chen2025}]\label{le12}
For $d=b+1$ and odd integer $d\geq1$,
$\lambda_{b+1}(G^{\star}_4(d,b))=\alpha_d,$ yet $G^{\star}_4(d,b)$ is not $\frac{1}{b}$-tough, where $\alpha_d$ is the largest root of the equation $x^3-(d-2)x^2-2dx+d-1=0$.
\end{lem}

Particularly, by the proof of Claim \ref{claim1} in Theorem \ref{thm4}, one can observe that the upper bound $d-b$ of $e(S,G_i)$ may be reduced for $s\geq2$.
So it is important and interesting to pose the following problem.

\begin{prob}\label{prob1}
What is a best possible upper bound of $\lambda_{b+1}(G)$ to guarantee a $2$-connected $d$-regular graph to be $\frac{1}{b}$-tough for $d\geq b+2,$ where $b$ is a positive integer?
\end{prob}

Let $a$ and $b$ be two positive integers with $a\leq b.$ A spanning subgraph $F$ is called an {\it $[a,b]$-factor} of $G$ if $a\leq d_F(v)\leq b$ for any vertex $v\in V(G).$
An $[a,b]$-factor is called an even (or odd) $[a,b]$-factor if $d_F(v)$ is even (or odd).
In 1970, Lov\'{a}sz\cite{Lovasz1970} presented a good characterization on the existence of parity $(g,f)$-factors in graphs.
Using the Lov\'{a}sz's parity $(g,f)$-factor Theory, O\cite{O2022} proved upper bounds on certain eigenvalues in an $h$-edge-connected $d$-regular graph $G$ to guarantee the existence of an even (or odd) $[a, b]$-factor. Meanwhile, O\cite{O2022} constructed graphs to ensure that the upper bounds are best possible. Note that an even (or odd) $[a, b]$-factor is a special $[a,b]$-factor. Furthermore, O\cite{O2022} proposed the following meaningful and interesting question.

\begin{prob}\label{prob2}
What is a best possible upper bound for a certain eigenvalue that can ensure an $h$-edge-connected $d$-regular graph contains a (connected) $[a, b]$-factor, where $a$ and $b$ are two positive integers with $a\leq b$?
\end{prob}
\vspace{3mm}
\noindent
{\bf Declaration of competing interest}
\vspace{3mm}

The authors declare that they have no known competing financial interests or personal relationships that could have appeared to influence the work reported in this paper.

\vspace{5mm}
\noindent
{\bf Data availability}
\vspace{3mm}

No data was used for the research described in this paper.

%\noindent
%{\bf Acknowledgement}

%The authors would like to thank the anonymous referees for their helpful comments on improving the presentation of
%paper.

\end{document}